\documentclass[12pt,a4paper]{amsart}
\usepackage{amsmath}
\usepackage{amsfonts}
\usepackage{amssymb}
\usepackage{subfigure}
\usepackage{graphicx}

\newcommand{\ceil}[1]{\ensuremath{\lceil #1 \rceil}}
\newcommand{\mes}[1]{{|#1|}}
\newcommand{\Rdst}{{\mathbb{R}^d}}
\newcommand{\Rst}{{\mathbb{R}}}

\newcommand{\Rtdst}{{\mathbb{R}^{2d}}}

\newcommand{\LtRd}{{L^2(\Rdst)}}
\newcommand{\LiRd}{{L^1(\Rdst)}}
\newcommand{\LtRtd}{{L^2(\Rtdst)}}
\newcommand{\norm}[1]{\lVert#1\rVert}
\newcommand{\bC}{{\mathbb{C}}}
\newcommand{\set}[2]{\big\{ \, #1 \, :  \, #2 \, \big\}}

\newcommand{\bignorm}[1]{\bigl\lVert#1\bigr\rVert}

\newcommand{\ip}[2]{\ensuremath{\left<#1,#2\right>}}
\newcommand{\sett}[1]{\ensuremath{\left \{ #1 \right \}}}
\newcommand{\abs}[1]{\ensuremath{\left| #1 \right| }}
\newcommand{\bigabs}[1]{\ensuremath{\big| #1 \big| }}
\newcommand{\Bigabs}[1]{\ensuremath{\Big| #1 \Big| }}

\newcommand{\disk}{{\mathbb{D}}}
\newcommand{\border}{\partial}
\newcommand{\error}{E}
\newcommand{\eigenf}{h}
\newcommand{\eigf}{H}

\newcommand{\trace}{\mathrm{trace}}
\newcommand{\mpluso}{\mathcal{M}_\Omega}
\newcommand{\env}{\Theta}
\newcommand{\VR}{{\mathbb{H}}}

\newcommand{\stft}{V}
\newcommand{\perim}[1]{\abs{\border{#1}}}
\newcommand{\inten}{\rho}

\newcommand{\lambdako}{{\lambda_k^\Omega}}
\newcommand{\lambdakro}{{\lambda_k^{R\cdot\Omega}}}

\newcommand{\BV}{\textit{BV}}
\newcommand{\var}{\textit{Var}}
\newcommand{\modstar}{{M^*}}
\newcommand{\modstarsp}{{M^*}}

\newcommand{\locop}{H}
\newcommand{\locom}{{\locop_\Omega}}

\newcommand{\divsymb}{\operatorname{div}}

\newtheorem{theo}{Theorem}[section]
\newtheorem{lemma}[theo]{Lemma}
\newtheorem{coro}[theo]{Corollary}
\newtheorem{prop}[theo]{Proposition}

\newtheorem{rem}[theo]{Remark}
\newtheorem{definition}[theo]{Definition}

\newtheorem*{rep@theorem}{\rep@title}
\newcommand{\newreptheorem}[2]{
\newenvironment{rep#1}[1]{
 \def\rep@title{#2 \ref{##1}}
 \begin{rep@theorem}}
 {\end{rep@theorem}}}

\newreptheorem{theorem}{Theorem}

\title{On accumulated spectrograms}
\author{Lu\'{\i}s Daniel Abreu}
\address{Acoustics Research Institute, Austrian Academy of Science, Wohllebengasse 12-14 A-1040, Vienna
Austria}
\email{daniel@mat.uc.pt}
\author{Karlheinz Gr\"{o}chenig}
\address{Faculty of Mathematics \\ University of Vienna \\ Oskar-Morgenstern-Platz 1 \\ A-1090 Vienna, Austria}
\email{karlheinz.groechenig@univie.ac.at}
\author{Jos\'e Luis Romero}
\email{jose.luis.romero@univie.ac.at}
\date{}
\thanks{L. D. A. was supported by the Austrian Science Fund (FWF) START-project FLAME
("Frames and Linear Operators for Acoustical Modeling and Parameter Estimation") 551-N13.
K.\ G. was  supported in part by the project P26273-N25 and the National Research Network S106 SISE of the
Austrian Science Fund (FWF). J. L. R. gratefully acknowledges support by the project M1586-N25 
of the Austrian Science Fund (FWF)}

\begin{document}
\begin{abstract}
  We study  the eigenvalues and  eigenfunctions of the time-frequency localization operator
  $H_\Omega $  on a domain $\Omega $ of the time-frequency plane.  The
  eigenfunctions are the appropriate prolate spheroidal functions for
  an arbitrary domain $\Omega \subseteq \Rtdst$. Indeed, in analogy to
  the classical theory of Landau-Slepian-Pollak, the number of
  eigenvalues of $H_\Omega $ in $[1-\delta , 1]$ is equal to the
  measure of $\Omega $ up to an error term depending on the perimeter
  of the boundary of $\Omega $. Our main results show that the
  spectrograms of the eigenfunctions corresponding to the large
  eigenvalues (which we call the accumulated spectrogram)  form an approximate partition of
  unity of the given domain $\Omega $. We derive  asymptotic, 
  non-asymptotic, and weak-$L^2$ error estimates for  the accumulated
  spectrogram. As a consequence the domain $\Omega $ can be
  approximated solely from the spectrograms of eigenfunctions without
  information about their phase.  
\end{abstract}
\maketitle

\section{Introduction and results}
\subsection{The time-frequency localization problem}
The short-time Fourier transform of a function $f \in \LtRd$ with respect to a window 
$g \in \LtRd$, $\norm{g}_2=1$, is defined as
\begin{align}
\label{eq_stft}
V_{g}f(z)
=\int_{\mathbb{R}^d}f(t)\overline{g(t-x)}e^{-2\pi i\xi t}dt,
\quad z=(x,\xi) \in \Rdst \times \Rdst.
\end{align}
The number $V_{g}f(x,\xi)$ quantifies the importance of the frequency $\xi$ of $f$ near $x$. The spectrogram
of $f$ is defined as  $\abs{V_g f}^2$ and measures  the distribution of the time-frequency content of $f$. 
The spectrogram is often interpreted as an energy density in time-frequency space. Its size depends on the window $g$.
The usual choice for $g$ is the Gaussian, because it provides optimal resolution in both time and frequency.

The uncertainty principle in Fourier analysis, in its several versions, sets a limit to the possible simultaneous
concentration of a function and its Fourier transform. In terms of the spectrogram, the uncertainty principle can be
roughly recast as follows: if a function $f$ has a spectrogram that
is essentially concentrated inside a region $\Omega \subseteq \Rtdst$, then the
area of $\Omega$ must be at least
1 (see for example the recent survey \cite{rito12}). In fact, if $f \in L^2(\Rdst)$ has norm 1 and
$\int_\Omega \abs{V_g f(z)}^2 dz \geq 1-\varepsilon$, then $\abs{\Omega} \geq 2^d(1-\varepsilon)^2$
\cite[Theorem 3.3.3]{gr01}.
Besides the basic restrictions on its measure, not much is known about the
possible shapes that such a set $\Omega$ can
assume.

In this article we choose a different point of view on the uncertainty
principle. We fix a compact domain $\Omega \subseteq \Rtdst$ in
time-frequency space and then try to determine those functions whose
spectrogram is essentially supported on $\Omega $.  Thus we try to   maximize the
concentration of the spectrogram of a function on a set $\Omega$.
To be precise, let $\Omega \subset \Rtdst$ be a compact set and $g\in
L^2(\mathbb{R}^d)$ a fixed window function. We  consider the following
optimization problem: 
\begin{align}
\label{eq_problem}
\mbox{ Maximize} \int_\Omega \abs{V_g f(z)}^2 \, dz,
\mbox{ with }\norm{f}_2=1.
\end{align}
In analogy to Landau-Pollack-Slepian theory of prolate spheroidal functions \cite{lapo61,lapo62,posl61}, this
problem can be studied through spectral analysis. The relevant operator is known as the 
\emph{time-frequency localization operator} with symbol $\Omega$
\cite{da88, da90} and is defined formally  as
\begin{align}
\label{eq_loc_op}
\locom f(t) = \int_{\Omega} V_{g}f(x,\xi) g(t-x) e^{2\pi i\xi t} dx d\xi,
\qquad t \in \Rdst.
\end{align}
It can be shown that if $\Omega$ is compact, then $\locom$ is a
compact and positive  operator on $L^2(\mathbb{R}^d) $ \cite{bocogr04,cogr03,duwozh01}. Hence $\locom$ can be
diagonalized as
\begin{align}
\label{eq_diag}
\locom f = \sum_{k \geq 1} \lambdako \ip{f}{\eigenf^\Omega_k}\eigenf^\Omega_k, \qquad f \in \LtRd,
\end{align}
where $\sett{\lambdako: k \geq 1}$ are the non-zero eigenvalues of $\locom$ ordered non-increasingly
and $\sett{\eigenf^\Omega_k: k \geq 1}$ is the corresponding orthonormal set of eigenfunctions.
(The functions $\eigenf^\Omega_k$ and the eigenvalues $\lambdako$ depend on the choice of the window $g$, but we do not
make this dependence explicit in the notation.)

The reason why $\locom$ is useful for studying the optimization problem \eqref{eq_problem} is that
\begin{align*}
\ip{\locom f}{f}=
\int_{\Omega} V_{g}f(x,\xi) \ip{g(\cdot-x) e^{2\pi i\xi \cdot}}{f} dx d\xi
=\int_{\Omega} \abs{V_g f(x,\xi)}^2 dxd\xi.
\end{align*}
Consequently, the first eigenfunction $\eigenf^\Omega_1$ of $\locom$
solves \eqref{eq_problem}:
\begin{align*}
\lambda^\Omega_1 = \ip{\locom \eigenf^\Omega_1}{\eigenf^\Omega_1}=
\int_{\Omega} \abs{V_g \eigenf^\Omega_1(z)}^2 dz=
\max \left\{\int_\Omega\abs{V_g f(z)}^2 dz: \norm{f}_2=1\right\}.
\end{align*}
If the set $\Omega$ is small, we expect $\lambda^\Omega_1=\int_{\Omega} \abs{V_g \eigenf^\Omega_1}^2$ to be small
because, as a consequence of the uncertainty principle, no spectrogram fits inside $\Omega$.
On the other hand, if $\Omega$ is big we expect \eqref{eq_problem} to have a number of approximate solutions, since a
number of spectrograms may fit inside $\Omega$. This intuition is made precise by studying the distribution of
eigenvalues of $\locom$. The min-max lemma for self-adjoint operators asserts that
\begin{align}
\label{eq_minimax}
\lambda^\Omega_k =
\int_{\Omega} \abs{V_g \eigenf^\Omega_k(z)}^2 dz=
\max \left\{\int_\Omega \abs{V_g f(z)}^2 dz: \norm{f}_2=1, f \perp \eigenf^\Omega_1, \ldots, \eigenf^\Omega_{k-1}
\right\}.
\end{align}
Hence, the profile of the eigenvalues of $\Omega$ shows how many orthogonal functions have a spectrogram
well-concentrated on $\Omega$. The standard asymptotic distribution  for the eigenvalues of $\locom$ involves dilating
a fixed set $\Omega$ and reads as follows.
\begin{prop}
\label{prop_eigen_soft}
Let $g \in L^2(\Rdst)$, $\norm{g}_2=1$, and
let $\Omega \subset \Rtdst$ be a compact set. Then for each $\delta \in (0,1)$,
\begin{align}
\label{eq_asympt}
\frac{\#\sett{k:\lambdakro > 1-\delta}}{\mes{R \cdot \Omega}} \longrightarrow 1,
\mbox{ as }R \longrightarrow +\infty.
\end{align}
\end{prop}
Proposition \ref{prop_eigen_soft} was proved with additional assumptions on the boundary of $\Omega$ by Ramanathan and
Topiwala \cite{rato94} and in full generality by Feichtinger and Nowak \cite{feno01}. For sets with smooth
boundary, more refined asymptotics are available \cite{defeno02, feno01, herato94, herato97, rato94} (see also
Section
\ref{sec_eigenvalues}). These results parallel the fundamental results for Fourier multipliers (ideal low-pass filters)
by Landau, Pollak and Slepian
\cite{la67-1,la67,la75-1,lapo61,lapo62,lawi80,posl61}.  Indeed, the
eigenfunctions of $\locom $ are the proper analogue of the prolate
spheroidal functions  associated to a general domain $\Omega $ in
phase space. 

See Figure
\ref{fig_eigenvals} for a numerical example for
Proposition \ref{prop_eigen_soft} for a star-shaped domain.

\begin{figure}
\centering
\subfigure[A domain with the shape of a star and area $\approx$ 23.]{
\includegraphics[scale=0.8]{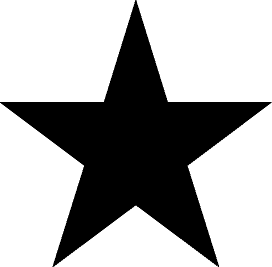}
\label{fig_star}
}
\subfigure[A plot of the eigenvalues illustrating Proposition 
\ref{prop_eigen_soft}.
Note that the star domain has considerable perimeter in relation to its area]{
\includegraphics[scale=0.4]{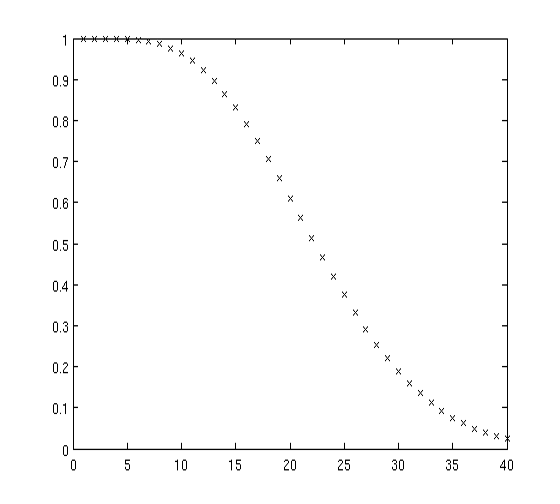}
}
\caption{The eigenvalues of a time-frequency localization operator with Gaussian window.}
\label{fig_eigenvals}
\end{figure}
\subsection{Accumulation of spectrograms}
The asymptotic behavior of the eigenvalue distribution  in
\eqref{eq_asympt} implies  that, after sufficiently dilating $\Omega$,
the $L^2$ concentration $\int_\Omega \abs{V_g h^\Omega_k(z)}^2 \, dz$ is close to 1 for $1 \leq k \leq n \approx
\mes{\Omega}$
and decays for large $k$. The purpose of this article is to refine the description of the time-frequency localization
of the eigenfunctions $h^\Omega_k$. We will show that the corresponding spectrograms 
approximately form a partition of unity on $\Omega$.

More precisely, we consider the following function, which we call the \emph{accumulated spectrogram}.
\begin{definition}
For a compact set $\Omega\subseteq \Rtdst$ and a window function $g \in \LtRd$ we let
$A_\Omega:=\ceil{\mes{\Omega}}$ be the smallest integer greater than or equal to $\mes{\Omega}$ and
$\sett{h_k^\Omega: k \geq 1}$ be the set of normalized eigenfunctions of $\locom$ ordered non-increasingly
with respect to the corresponding eigenvalues. The accumulated spectrogram of $\Omega$ 
(with respect to $g$) is
\begin{align*}
&\inten_\Omega(z) := \sum_{k=1}^{A_\Omega} \abs{V_g
  h^\Omega_k(z)}^2,\qquad z\in \Rtdst \, .
\end{align*}
\end{definition}
Our goal is to prove that $\inten_\Omega$ looks approximately like $1_\Omega$
(the characteristic function of $\Omega$). Since $0 \leq \abs{V_g h^\Omega_k(z)}^2 \leq 1$,
this means that the spectrograms
$\abs{V_g h^\Omega_1}^2, \ldots, \abs{V_g h^\Omega_{A_\Omega}}^2$ form an approximate partition of unity on $\Omega$.
Indeed, numerical experiments show that
$\inten_\Omega$ resembles a bump function on $\Omega$ plus a tail around its boundary (see Figure
\ref{fig_inten}).
The size of this tail grows when $\Omega$ grows but at a smaller rate than $\mes{\Omega}$. Our main results
will validate these observations. 

\begin{figure}
\label{fig_intens}
\centering
\subfigure[The accumulated spectrogram plotted over the domain.]{
\includegraphics[scale=0.4]{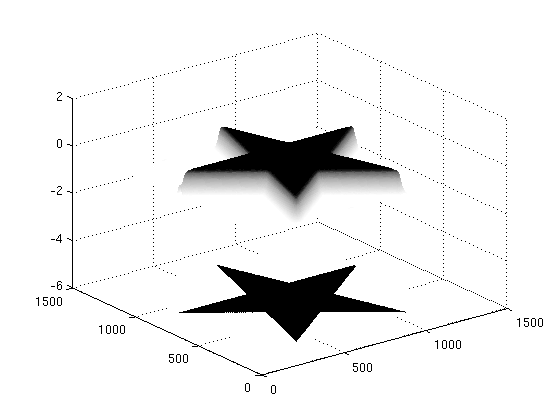}
\label{fig_starfl}
}
\subfigure[A cross-section of the domain and the accumulated spectrogram.]{
\includegraphics[scale=0.4]{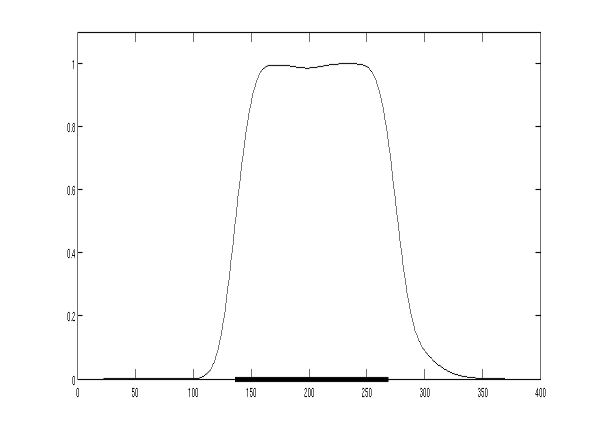}
\label{fig_cross}
}
\caption{A domain with the form of a star.}
\label{fig_inten}
\end{figure}

To argue that the tail in $\inten_\Omega$ is asymptotically smaller than $\mes{\Omega}$ we consider dilations
of a fixed set $\Omega$ and rescale $\inten_{R\cdot \Omega}$ by a factor of $R$. We then have the
following result.
\begin{theo}
\label{th_one_point}
Let $g \in L^2(\Rdst)$, $\norm{g}_2=1$, and let $\Omega \subset \Rtdst$ be compact. Then
\begin{align*}
\inten_{R\cdot\Omega} (R \cdot) \longrightarrow 1_\Omega
\mbox{ in $L^1(\Rtdst)$, as }R\longrightarrow +\infty.
\end{align*}
\end{theo}
Under mild regularity assumptions on $g$ and $\Omega$ we give a more
quantitative and  
non-asymptotic estimate that explains the tail in
$\inten_\Omega$ as an effect of the window $g$. To quantify the statement in
Theorem \ref{th_one_point}, we assume that
$g$ satisfies the following time-frequency concentration condition:
\begin{align}
\label{eq_modstar}
\norm{g}_\modstar^2 := \int_\Rtdst \abs{z} \abs{V_g g(z)}^2 dz < +\infty.
\end{align}
We let $\modstarsp(\Rdst)$ denote the class of all $L^2(\Rdst)$ functions satisfying \eqref{eq_modstar}.
It follows easily that the Schwartz class is contained in $\modstarsp(\Rdst)$.
(The class $\modstarsp$ is closely related to the \emph{modulation spaces}
$M^1_{1}(\Rdst)$ and $M^2_{1/2}(\Rdst)$.  See Remark~\ref{rem_mod}).

In addition, we assume that $\Omega$ has \emph{finite perimeter}. This means that its characteristic function $1_\Omega$
is of bounded variation.  Every  compact set $\Omega
\subseteq \Rtdst$ with smooth boundary has a  finite
perimeter and its perimeter is the $(2d-1)$-dimensional surface measure of its topological boundary.
(Section \ref{sec_bv} below provides more  background on functions of bounded variation and sets of finite perimeter.)
\begin{theo}
\label{th_non_asym}
Assume that $g \in \modstarsp(\Rdst)$ with $\norm{g}_2=1$
and that $\Omega \subset \Rtdst$ is a compact set with finite perimeter.
Then
\begin{align*}
\frac{1}{\mes{\Omega}} 
\bignorm{\inten_{\Omega}- 1_{\Omega}*\abs{V_g g}^2}_1 \leq 
\left(
\frac{1}{\mes{\Omega}}+4\norm{g}_\modstar \sqrt{\frac{\perim{\Omega}}{\mes{\Omega}}}
\right),
\end{align*}
where $\perim{\Omega}$ is the perimeter of $\Omega$.
\end{theo}

For $\mes{\Omega} \geq 1$, Theorem \ref{th_non_asym} implies the  weaker estimate
\begin{align}
\label{eq_C}
\frac{1}{\mes{\Omega}}
\bignorm{\inten_{\Omega}- 1_{\Omega}}_1 \leq C
\sqrt{\frac{\perim{\Omega}}{\mes{\Omega}}},
\end{align}
where the constant $C$ depends only on the window $g$ (see Corollary \ref{coro_error}).
Since  $\norm{\inten_{\Omega}}_1 = A_\Omega = \mes{\Omega} + O(1)$
and $\norm{1_{\Omega}}_1=\mes{\Omega}$,
the $L^1$-difference between $\inten_{\Omega}$ and $1_\Omega$ is much
smaller than the norm of these two functions.

While \eqref{eq_C} bounds the  $L^1$-error (normalized over $\Omega $)
incurred when approximating $\inten_\Omega$ by
$1_\Omega$, it is even more  interesting to obtain pointwise error
estimates. Applying Chebyshev's Inequality to~\eqref{eq_C}, we obtain
that 
$$
\Bigabs{\sett{ z\in \mathbb{R}^{2d} : \bigabs{\inten_\Omega (z)
      -1_\Omega(z) }>\delta}}
\lesssim \frac{ \big( \perim{\Omega} |\Omega | \big)^{1/2} }{\delta}
\, .
$$
As our final result 
we will prove the following refined weak-$L^2$ error estimate.

\begin{theo}
\label{th_weak_l2}
Let $g \in \modstarsp(\Rdst)$ and let $\Omega \subset \Rtdst$ be a compact set with finite perimeter
and assume that $\norm{g}^2_\modstarsp \perim{\Omega} \geq 1$. Then
\begin{align*}
\bigabs{\sett{ z\in \mathbb{R}^{2d} : \bigabs{\inten_\Omega (z)
      -1_\Omega(z) }>\delta}}
\lesssim \frac{1}{\delta^2}\norm{g}_{\modstar}^2 \perim{\Omega},
\qquad \delta>0.
\end{align*}
\end{theo}

Theorem~\ref{th_weak_l2} says that the size of the set where $\rho
_\Omega $ differs significantly from $1_\Omega $ is of order
$\perim{\Omega }$. This is the expected order: along the boundary
$\partial \Omega $ the characteristic function $1_\Omega $  has a jump
singularity, whereas the accumulated spectrogram $\rho _\Omega $ is
smooth (at least uniformly continuous). Therefore $1_\Omega - \rho _\Omega $ must be large in a neighborhood of
$\partial \Omega $, which suggests that $\abs{\{ z\in \mathbb{R}^{2d} : \bigabs{\inten_\Omega (z)
      -1_\Omega(z) }>\delta}\} \geq c_\delta  \perim{\Omega}$.

For a numerical illustration of Theorem~\ref{th_weak_l2} see
Figure~\ref{fig_inten}.

\begin{figure}
\centering
\subfigure[A domain of area $\approx$ 23, the spectrograms of the
eigenfunctions for $k=1 , 7$ and $17$, 
and the accumulated spectrogram with respect to the Gaussian $g(t)=2^{1/4} e^{-\pi t^2}$.
The shading of each eigenfunction is relative to its own scale of values.
]{
\includegraphics[scale=1]{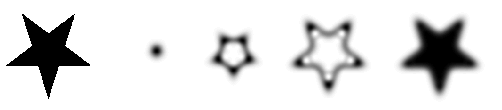}
}
\subfigure[The same experiment with the Gaussian window
$g(t)=2^{-1/4} e^{-\pi (t/2)^2}$, that is
is more concentrated in frequency than the Gaussian in (a)]{
\includegraphics[scale=1]{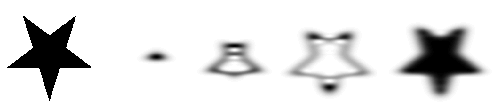}
}
\caption{The first eigenfunctions of a time-frequency localization
  operator with identical domain $\Omega $, but  with different
  windows. One obtains  two different (almost) partitions of unity on the same domain.}
\label{fig_funcs}
\end{figure}

\subsection{Universality}
In general, the  individual eigenfunctions and their spectrograms are difficult to
describe, since they  depend  on the underlying window $g$ in an
intricate manner (see Figure \ref{fig_funcs}).
Nevertheless, by \eqref{eq_C}   the spectrograms of the whole family $\{h^\Omega_1, \ldots, h^\Omega_{A_\Omega}\}$
almost form  a  partition of unity on $\Omega$. By and large this
property is independent of the window $g$ (which enters only through
some constants in the error estimates). Even more is true: the
asymptotic result of Theorem~\ref{th_one_point} is completely independent of
$g$.

In mathematics, the term ``universality''  is used   to describe phenomena
where   properties of complex  systems simplify in an asymptotic
regime  and then depend
only on few parameters. The primary examples are  the central limit
theorem in  probability or the eigenvalue distribution of random
matrices. (See \cite{de07-7} for
an overview of  universality in several contexts.)
 
In this sense,  Theorem~\ref{th_one_point} expresses a new and intriguing
universality property.
Whereas the accumulated spectrograms for a fixed  domain with different
windows may display a vastly different behavior for small scales, they
all approach the characteristic function in the scaling limit $R\to
\infty $. See Figure~\ref{fig_funcs}.

The universality property expressed by Theorem 1.3 has also a
probabilistic interpretation. Consider the finite dimensional space
generated by the short-time Fourier transform of the first $A_\Omega$
eigenfunctions of $H_\Omega$ and its associated reproducing kernel. The
accumulated spectrogram is the  one-point intensity of the determinantal
point process  associated with this  reproducing kernel (see \cite{DetPointRand}
for precise definitions). Thus, Theorem 1.3 asserts  that the asymptotic limit of this one-point
intensity is the uniform distribution on  $\Omega$,  independently
of the chosen window g.

\subsection{Ginibre's law}
It is instructive to discuss a case where all the objects can be computed
explicitly. Let $g(t) := 2^{1/4} e^{-\pi t^2}$, $t \in \Rst$,  be the
one-dimensional, $L^2$-normalized Gaussian and 
let \[\disk:= \set{(x,\xi) \in \Rst^2}{x^2 + \xi^2 \leq 1}\] be the
unit disk. By a result of  Daubechies \cite{da88} 
the eigenfunctions and eigenvalues of the time-frequency localization operator with window $g$ and domain
$\Omega= R\cdot \disk$ for arbitrary $R>0$  are the Hermite functions
\begin{align}
\label{eq_hermite}
h_{k+1}(t) = \frac{2^{1/4}}{\sqrt{k!}}\left(\frac{-1}{2\sqrt{\pi}}\right)^k
e^{\pi t^2} \frac{d^k}{dt^k}\left(e^{-2\pi t^2}\right), \qquad k \geq 0.
\end{align}
Remarkably, due to the symmetries of $g$ and $\disk$, the
eigenfunctions $h_k$ of $\locop_{R\disk }$ do not depend on $R$.
Identifying $z=(x,\xi ) \in \mathbb{R}^{2} $ with  $z := x + i \xi \in
\bC$, the short-time Fourier transform of the Hermite functions with
respect to $g$ is 
\begin{align*}
V_g h_{k+1}(\bar{z}) = e^{\pi i x \xi} \left(\frac{\pi^k}{k!}\right)^{1/2} z^k e^{-\pi\abs{z}^2/2},
\qquad k \geq 0.
\end{align*}
Hence the corresponding spectrograms are
\begin{align*}
\abs{V_g h_{k+1}(z)}^2 = \frac{\pi^k}{k!} \abs{z}^{2k} e^{-\pi\abs{z}^2},
\qquad k \geq 0, 
\end{align*}
and the accumulated spectrogram corresponding to $R\cdot\disk$ is
\begin{align*}
\inten_{R\cdot \disk}(z) = \sum_{k=0}^{\ceil{\pi R^2}-1}
\frac{\pi^k}{k!} \abs{z}^{2k} e^{-\pi\abs{z}^2}.
\end{align*}
Theorem \ref{th_one_point} says that
\begin{align*}
\inten_{R\cdot \disk}(Rz)=
\sum_{k=0}^{\ceil{\pi R^2}-1}
\frac{\pi^k}{k!} R^{2k} \abs{z}^{2k} e^{-\pi\abs{Rz}^2}
\longrightarrow 1_\disk(z),
\mbox{ in }L^1(\bC, dxd\xi)\mbox{ as }n \longrightarrow +\infty.
\end{align*}
This formula is one of Ginibre's limit laws \cite{gi65}. (Of course,
in this case more precise asymptotic estimates  are
possible.) 

\subsection{Technical overview}
As commonly done in the literature, we transfer the problem to the range of
the short-time Fourier transform $V_g L^2(\Rdst)$. This is a reproducing kernel subspace of $L^2(\Rtdst)$
and the time-frequency localization operator $\locom$ translates into a Toeplitz operator on $V_g L^2(\Rdst)$.

Our proof begins with the observation that $\inten_\Omega$ is the diagonal of the integral kernel
of the orthogonal projection operator $P_n$ from $V_g L^2(\Rdst)$ onto the linear span of the functions
$\sett{V_g h^\Omega_1, \ldots, V_g h^\Omega_n}$. We then  compare $P_n$ to the Toeplitz operator,
which has an explicit integral kernel for which we derive direct estimates. The relation
between the Toeplitz operator and $P_n$ can be understood as a threshold on its eigenvalues. We then resort
to the rich existing literature  on these
\cite{defeno02, feno01, herato94, herato97, la67-1, la67, la75-1, lawi80, rato94}.

\subsection{Organization}
Section \ref{sec_tools} introduces some basic tools from time-frequency analysis, including the
interpretation of time-frequency localization operators as Toeplitz operators on a certain reproducing kernel space.
In Section \ref{sec_eigenvalues} we recall the estimates on the profile of the eigenvalues of
time-frequency localization operators, and add a slight refinement. Section \ref{sec_bounds} develops a number of
technical estimates that are used in Section \ref{sec_res} to prove the main results. Finally, Section
\ref{sec_app} briefly presents an application of Theorem \ref{th_weak_l2} to signal processing.

\section{Phase-space tools}
\label{sec_tools}
We now introduce the basic phase-space tools related to the short-time Fourier transform.
\subsection{The range of the short-time Fourier transform}
Throughout the article we fix a window $g \in L^2(\Rdst)$ with normalization
$\norm{g}_2=1$. The short-time Fourier transform with respect to $g$ - cf. \eqref{eq_stft} - then defines an
isometry $V_g:L^2(\Rdst) \to L^2(\Rtdst)$.
(This is a consequence of Plancherel's theorem, see for example \cite[Chapter 1]{gr01}).
The adjoint of $V_g$ is $V^*_g:L^2(\Rtdst) \to L^2(\Rdst)$
\begin{align*}
V^*_g F(t) = \int_{\Rdst \times \Rdst} F(x,\xi) g(t-x)e^{2\pi i\xi t} dx d\xi,\qquad t\in\Rdst.
\end{align*}
Let $\VR:=\stft_g L^2(\Rdst) \subseteq L^2(\Rtdst)$ be the (closed) range of $V_g$. The orthogonal projection
$P_\VR: L^2(\Rtdst) \to \VR$ is given by $P_\VR = V_g V^*_g$. Explicitly, 
$P_\VR$ is the integral operator
\begin{align}
\label{eq_proj}
P_\VR F(z) &= \int_\Rtdst F(z') K(z,z') dz',
\qquad z=(x,\xi) \in \Rtdst,
\end{align}
with integral kernel
\begin{align}
\label{eq_rep_ker}
K(z,z')=\stft_g g(z-z') e^{2\pi i \xi x'},
\qquad z=(x,\xi), z'=(x',\xi') \in \Rtdst.
\end{align}
Using this description of $P_\VR$ it follows that $\VR$ is a \emph{reproducing kernel Hilbert subspace} of $\LtRtd$.
This means that each function $F \in \VR$ is continuous and satisfies
$F(z)=\int F(z') K(z,z') dz'$ for all $z \in \Rtdst$. The function $K$ is called the \emph{reproducing kernel} of
$\VR$.

Since $K$ is the integral kernel of an orthogonal projection, it satisfies
$K(z,z')=\overline{K(z',z)}$ and
\begin{align}
\label{eq_kernel_square}
K(z,z') = \int_\Rtdst K(z,z'')K(z'',z') dz'', \qquad z,z' \in \Rtdst.
\end{align}
In addition, if $\sett{E_k: k\geq 1}$ is an orthonormal basis of $\VR$, $K$ can be expanded as
\begin{align}
\label{eq_exp_ker}
K(z,z')=\sum_{k \geq 1} E_k(z) \overline{E_k(z')}, \qquad z,z' \in \Rtdst.
\end{align}
From now on, we use the notation
\begin{align*}
\env(z) := \abs{\stft_g g(z)}^2, \qquad z \in \Rtdst.
\end{align*}
Then $\env \in L^1(\Rtdst)$, $\int _{\mathbb{R}^{2d}} \env = \int
_{\mathbb{R}^{2d}} |V_gg|^2 = \|g\|_2^4 = 1$, $\env(z)=\env(-z)$ and 
\begin{align}
\label{eq_abs_ker}
\abs{K(z,z')}^2=\env(z-z'), \qquad z,z' \in \Rtdst.
\end{align}
\subsection{Time-frequency localization and Toeplitz operators}
\label{sec_tf_loc}
Using $V_g$ and $V^*_g$,
the time-frequency localization operator $\locom:L^2(\Rdst) \to L^2(\Rdst)$ from 
\eqref{eq_loc_op} can be written as
\begin{align*}
\locom f = V^*_g (1_\Omega \cdot V_g f), \qquad f \in \LtRd.
\end{align*}
Therefore,
\begin{align}
\label{eq_vgm}
(V_g \locom V^*_g) F = P_\VR (1_\Omega \cdot F), \qquad F \in \VR.
\end{align}
Hence, if we define the \emph{Gabor-Toeplitz operator} $M_\Omega: \VR \to \VR$ as
\begin{align*}
M_\Omega F := P_\VR ( 1_{\Omega} \cdot F ),
\qquad F \in \VR,
\end{align*}
then \eqref{eq_vgm} says that $\locom$ and $M_\Omega$ are related by
\begin{align}
\label{eq_conj}
V^*_g M_\Omega V_g = \locom.
\end{align}
As a consequence, $M_\Omega$ and $\locom$ enjoy the same spectral properties. Using the
diagonalization of $\locom$ in \eqref{eq_diag} and the notation
\begin{align*}
\eigf^\Omega_k = V_g \eigenf^\Omega_k, \qquad k \geq 1,
\end{align*}
$M_\Omega$ can be diagonalized as
\begin{align}
\label{eq_MO}
M_\Omega F = \sum_{k \geq 1} \lambdako \ip{F}{\eigf^\Omega_k}\eigf^\Omega_k, \qquad F \in \VR.
\end{align}
In addition, the accumulated spectrogram can be written as
\begin{align*}
\inten_\Omega (z) = \sum_{k=1}^{A_\Omega} \abs{\eigf^\Omega_k(z)}^2,
\qquad z \in \Rtdst.
\end{align*}
Since $\sett{\eigf^\Omega_k: k \geq 1}$ is an orthonormal subset of $\VR$ and $K(z,z)=1$, 
it follows from \eqref{eq_exp_ker} that
\begin{align}
\label{eq_sum_one}
\sum_{k \geq 1} \abs{\eigf^\Omega_k(z)}^2 \leq 1, \qquad z \in \Rtdst,
\end{align}
and consequently, the accumulated spectrogram satisfies 
\begin{align}
\label{eq_inten_bound}
0 \leq \inten_\Omega (z) \leq 1, \qquad z \in \Rtdst.
\end{align}
\subsection{Properties of Toeplitz operators}
\label{sec_toep}
Using \eqref{eq_proj} we see that $M_\Omega:\VR \to \VR$ can be described as
\begin{align*}
M_\Omega F(z) = \int_\Omega F(z') K(z,z') dz', \qquad z \in \Rtdst.
\end{align*}
We now note some simple properties of Toeplitz operators. These are well-known
\cite{bocogr04, cogr03, defeno02, duwozh01, en09, feno01}, but are normally stated in a slightly different form.
Therefore we sketch a short proof.
\begin{lemma}
\label{lemma_toep}
Let $g \in L^2(\Rdst)$ with $\norm{g}_2=1$ and let $\Omega \subseteq \Rtdst$ be a compact set. Then
$M_\Omega: L^2(\Rdst) \to L^2(\Rdst)$ is trace-class and satisfies
\begin{align}
\label{eq_spec}
0 \leq M_\Omega \leq I.
\end{align}
The traces of $M_\Omega$ and $M_\Omega^2$ are given by
\begin{align}
\label{eq_trace1}
&\trace(M_\Omega) = \int_\Omega K(z,z) dz = \mes{\Omega},
\\
\label{eq_trace2}
&\trace(M_\Omega^2)= \int_\Omega \int_\Omega \env(z-z') dz dz'.
\end{align}
\end{lemma}
\begin{proof}
For $F \in \VR$ we compute
\begin{align*}
\ip{M_\Omega F}{F} = \ip{P_\VR ( 1_{\Omega} F )}{F}=\ip{1_\Omega F}{F}
= \int_\Omega \abs{F(z)}^2 dz.
\end{align*}
This gives \eqref{eq_spec}, and in particular shows that $M_\Omega$ is bounded and positive.

Let $\mpluso: L^2(\Rtdst) \to L^2(\Rtdst)$ be the \emph{inflated} operator
$\mpluso(F) := P_\VR (1_\Omega \cdot P_\VR(F))$. Hence, with respect to the decomposition
$L^2(\Rtdst) = \VR \oplus \VR^\perp$, $\mpluso$ is given by
\begin{align}
\label{eq_mplus}
\mpluso = \left[
\begin{array}{ll}
 M_\Omega & 0\\
 0 & 0
\end{array}\right].
\end{align}
Since $\mpluso$ is defined on $L^2(\Rtdst)$, its spectral properties can be easily related to its
integral kernel. Using
\eqref{eq_proj} we get the following formula for $\mpluso$
\begin{align*}
\mpluso F(z) =
\int_\Rtdst F(z') \int_\Rtdst 1_\Omega(z'') K(z,z'') K(z'',z') dz'' dz',
\qquad F \in L^2(\Rtdst).
\end{align*}
That is, $\mpluso$ has integral kernel
$K_{\mpluso}(z,z') := \int 1_\Omega(z'') K(z,z'') K(z'',z') dz''$.

Using \eqref{eq_abs_ker} we compute
\begin{align*}
\int_\Rtdst K_{\mpluso}(z,z) dz=
\int_\Rtdst \int_\Rtdst 1_\Omega(z') \env(z-z') dz dz' = \mes{\Omega} \norm{\env}_1
= \mes{\Omega}.
\end{align*}
Since $M_\Omega$ is positive, so is $\mpluso$ by \eqref{eq_mplus}. Hence, the previous calculation
shows that $\mpluso$ is trace-class and $\trace(\mpluso)=\mes{\Omega}$
(see for example \cite[Theorems 2.12 and 2.14]{si79}). Using \eqref{eq_mplus} we deduce that
$M_\Omega$ is trace class and $\trace(M_\Omega)=\trace(\mpluso)=\mes{\Omega}$. For~\eqref{eq_trace2}, we use again the fact that $\mpluso$ is positive to get
\begin{align*}
&\trace(M_\Omega^2)=\trace(\mpluso^2)=\norm{K_{\mpluso}}_2^2
\\
&\quad=\int_{\Rtdst}\int_{\Rtdst}K_{\mpluso}(z,z')K_{\mpluso}(z',z) dz dz'
\\
&\quad=
\int_{\Rtdst} \int_{\Rtdst} 1_\Omega(w) 1_\Omega(w')
\int_{\Rtdst}\int_{\Rtdst}
K(z,w) K(w,z') K(z',w') K(w',z)dz dz' dw dw'.
\end{align*}
Using \eqref{eq_kernel_square} and \eqref{eq_abs_ker} it follows that
\begin{align*}
\trace(M_\Omega^2)&=
\int_{\Rtdst} \int_{\Rtdst} 1_\Omega(w) 1_\Omega(w') K(w,w') K(w',w)dw dw'
\\
&= 
\int_{\Rtdst} \int_{\Rtdst} 1_\Omega(w) 1_\Omega(w') \env(w-w') dw dw',
\end{align*}
as desired.
\end{proof}
\section{Profile of the eigenvalues}
\label{sec_eigenvalues}
Let $g \in \LtRd$, $\norm{g}_2=1$, and consider the time-frequency localization operator
$\locom$ associated with a compact set $\Omega$.
The fundamental asymptotic behavior  of the profile of the eigenvalues $\sett{\lambdako: k \geq 1}$
are given by Proposition \ref{prop_eigen_soft}, and were obtained by Ramanathan and Topiwala 
in a weaker form and by Feichtinger and Nowak in full generality \cite{feno01,rato94}.

For sets with smooth boundary, Ramanathan and Topiwala \cite{rato94} quantified the order of convergence in 
\eqref{eq_asympt}. Later De Mari, Feichtinger and Nowak obtained fine asymptotics for families of sets where certain
geometric quantities are kept uniform \cite{defeno02} - the dilations of a set with smooth boundary being one such an
example. Moreover their results are also applicable to the hyperbolic setting \cite{defeno02}.

We will need a variation of the result in \cite{rato94}. 
\subsection{Functions of bounded variation}
\label{sec_bv}
We first collect  some basic facts about functions of bounded
variation. A real-valued function $f \in L^1(\Rdst)$ is said to have
\emph{bounded variation}, $f\in \BV (\Rdst )$,   if its distributional partial derivatives
are finite Radon measures. The variation of $f$ is defined as 
\begin{align*}
\var(f) := \sup
\sett{\int_\Rdst f(x) \divsymb \phi(x) dx: \phi \in
  C^1_c(\Rdst,\Rdst), \abs{\phi (x)}_2 \leq 1} \, ,
\end{align*}
where $C^1_c(\Rdst,\Rdst)$ denotes the class of  compactly supported $C^1$-vector fields
and $\divsymb$ is the divergence operator.

If $f$ is continuously differentiable, then $f\in \BV (\Rdst )$ simply
means that
$\partial_{x_1}f, \ldots $,   $  \partial_{x_d} f \in L^1(\Rdst)$, and 
$\var(f) = \int_\Rdst \abs{\nabla f(x)}_2 dx$.

A set $E \subseteq \Rdst$ is said to have \emph{finite perimeter} if its characteristic function $1_E$ is of bounded
variation, and the  perimeter of $E$  is $\perim{E} := \var(1_E)$. If
$E$ has a smooth boundary, then   $\perim{E}$ is just  the $(2d-1)$-Hausdorff
measure of the topological boundary. See \cite[Chapter 5]{evga92} for an extensive
discussion of $\BV $. 

We will use  the following approximation result from ~\cite[Sec.~5.2.2, Thm.~2]{evga92}.
\begin{prop}
\label{prop_bv}
Let $f \in \BV(\Rdst)$. Then there exists $\sett{f_k: k\geq 1} \subseteq \BV(\Rdst) \cap C^\infty(\Rdst)$
such that $f_k \longrightarrow f$ in $L^1(\Rdst)$ and $\int \abs{\nabla f_k(x)}_2 dx \longrightarrow \var(f)$,
as $k \longrightarrow +\infty$.
\end{prop}

Note that this  proposition does not assert that $\var(f-f_k) \longrightarrow 0$.

The following lemma  quantifies the error introduced by
regularization.
\begin{lemma}
\label{lemma_var}
Let $f \in \BV(\Rdst)$ and $\varphi \in \LiRd$ with $\int \varphi = 1$. Then
\begin{align*}
\norm{f*\varphi-f}_1 \leq \var(f) \int_\Rdst \abs{x}_2 \abs{\varphi(x)} dx.
\end{align*}
\end{lemma}
\begin{proof}
Assume first that $f \in C^\infty(\Rdst)$ and estimate
\begin{align*}
&\norm{f*\varphi-f}_1 =
\int_\Rdst \abs{\int_\Rdst f(y)\varphi(x-y) \,dy - f(x)} dx
\\
&\qquad =
\int_\Rdst \abs{\int_\Rdst (f(y)-f(x))\varphi(x-y) \,dy} dx
\\
&\qquad =
\int_\Rdst \abs{\int_\Rdst (f(x)-f(y))\varphi(x-y) \,dy} dx
\\
&\qquad =
\int_\Rdst
\abs{\int_\Rdst \int_0^1 \ip{\nabla f(t(x-y)+y)}{x-y} \varphi(x-y) dt dy} dx
\\
&\qquad \leq 
\int_0^1 \int_\Rdst
\int_\Rdst \abs{\nabla f(t(x-y)+y)}_2 \abs{x-y}_2 \abs{\varphi(x-y)} dx dy dt
\\
&\qquad = 
\int_0^1 \int_\Rdst
\int_\Rdst \abs{\nabla f(tx+y)}_2 \abs{x}_2 \abs{\varphi(x)} dx dy dt
\\
&\qquad = 
\bignorm{\abs{\nabla f}_2}_1 \int_\Rdst \abs{x}_2 \abs{\varphi(x)} dx.
\end{align*}
For an arbitrary function of bounded variation $f$, 
Proposition \ref{prop_bv} implies the existence of a sequence
$\sett{f_k:k\geq 1} \subseteq BV(\Rdst) \cap C^\infty(\Rdst)$ such that
$f_k \longrightarrow f$ in $L^1$ and
$\bignorm{\abs{\nabla f_k}_2}_1 \longrightarrow \var(f)$. The desired estimate now follows by approximation.
\end{proof}

\subsection{Estimates on eigenvalues}
We now quote the standard estimate for the eigenvalue distribution.
This estimate is ubiquitous in the literature - see for example \cite[Thm.~1]{la75-1}, \cite[Thm.~2]{rato94} or
\cite[Remark (iii)]{feno01}. For lack of an exact quotable reference we sketch a proof
following \cite{feno01}.
\begin{lemma}
\label{lemma_trace_norm}
Let $g \in \LtRd$, $\norm{g}_2=1$, and let $\Omega \subset \Rtdst$ be compact. Then
for $\delta \in (0,1)$,
\begin{align*}
\Bigabs{\# \sett{k \geq 1: \lambdako > 1-\delta} - \mes{\Omega}}
\leq \max\sett{\frac{1}{\delta},\frac{1}{1-\delta}}
\abs{\int_{\Omega} \int_{\Omega} \env(z-z') dzdz' - \mes{\Omega}}.
\end{align*}
\end{lemma}
\begin{proof}
Let $\delta \in (0,1)$ and $G:[0,1]\to\Rst$ be the function
\begin{align*}
G(t) := \left\{
\begin{array}{ll}
-t, & \mbox{if } 0 \leq t \leq 1-\delta,\\
1-t, & \mbox{if } 1-\delta < t \leq 1.
\end{array}
\right.
\end{align*}
and note that $\abs{G(t)} \leq \max\sett{\frac{1}{\delta},\frac{1}{1-\delta}}
(t-t^2)$ for $t \in [0,1]$. Since $0 \leq M_\Omega \leq 1$ we estimate
\begin{align*}
&\Bigabs{\# \sett{k \geq 1: \lambdako > 1-\delta} - \mes{\Omega}}
=\abs{\trace(G(M_\Omega))} \leq \trace(\abs{G}(M_\Omega))
\\
&\qquad\leq \max\sett{\frac{1}{\delta},\frac{1}{1-\delta}} \trace(M_\Omega-M_\Omega^2).
\end{align*}
The formulas in Lemma \ref{lemma_toep} complete the proof.
\end{proof}

Finally, we obtain the following estimate, which is a variation of the result in \cite{rato94}. While \cite{rato94}
uses the co-area formula, we resort to Lemma \ref{lemma_var}.
\begin{prop}
\label{prop_eigen}
Let $\Omega \subset \Rtdst$ be a compact set with finite perimeter and assume that
$g \in \modstarsp(\Rdst)$. Then for each $\delta \in (0,1)$,
\begin{align*}
\Bigabs{ \#\sett{k:\lambdako > 1-\delta} - \mes{\Omega} }
\leq \max\sett{\frac{1}{\delta},\frac{1}{1-\delta}} \norm{g}_{\modstar}^2 \perim{\Omega}.
\end{align*}
\end{prop}
\begin{proof}
Recall that $\env(z)=\abs{V_gg(z)}^2=\env(-z)$, let $F:=1_\Omega$ and note that
\begin{align*}
\abs{\int_{\Omega} \int_{\Omega} \env(z-z') dzdz' - \mes{\Omega}}
= \abs{\int_\Omega (F*\env)(z') - F(z') dz'} \leq \norm{F*\env - F}_1.
\end{align*}
We now combine Lemmas \ref{lemma_var} and \ref{lemma_trace_norm} and
the desired conclusion follows.
\end{proof}
\begin{rem}
{\rm The results in \cite{defeno02} provide an asymptotic converse inequality to the estimate in Proposition
\ref{prop_eigen}, showing that
for certain classes of sets and small $\delta$, the error $\abs{ \#\sett{k:\lambdako > 1-\delta} -
\mes{\Omega} }$ is also bounded from below by $\perim{\Omega}$.  }
\end{rem}
\begin{rem}
\label{rem_mod}
{\rm The class $\modstarsp(\Rdst)$ is closely related to certain \emph{modulation spaces}.
Modulation spaces are Banach spaces of distributions $f$ defined by
decay  conditions of the short-time Fourier transform  $V_\phi f$ with
respect to a fixed Schwartz function $\phi$. 
In particular, the spaces $M^1_{1}(\Rdst)$ and $M^2_{1/2}(\Rdst)$ are defined 
by the norms
\begin{align*}
\norm{f}_{M^1_{1}} &:= \int_\Rdst \abs{z} \abs{V_\phi f(z)} dz,
\\
\norm{f}^2_{M^2_{1/2}} &:= \int_\Rdst \abs{z} \abs{V_\phi f(z)}^2 dz.
\end{align*}
It is then easy to see that $M^1_{1}(\Rdst) \subseteq \modstarsp(\Rdst) \subseteq M^2_{1/2}(\Rdst)$.
Easy sufficient condition for the membership in $M^1_1(\Rdst )$ and
$M^*(\Rdst )$ can be given by means of embeddings of Fourier-Lebesgue
spaces in modulation spaces~\cite{GG02}. For instance, if 
$$
\int _{\Rdst } \Big(|g(x)|^2 + |\hat{g} (x)|^2 \Big) \, (1+|x|^2)^{1/d
  + 1/2 +\epsilon } \, dx < \infty
$$
for some $\epsilon >0$, then $g\in M^1_1(\Rdst ) \subseteq M^*(\Rdst )$. 
For the theory of modulation spaces we refer to ~\cite{gr01}, for a
historical  overview  see \cite{fe06}.}
\end{rem}

\section{Some technical estimates}
\label{sec_bounds}
We now derive a number of lemmas. In the next section we combine them in several ways to obtain the
main results. Recall that we always  assume that $g \in L^2(\Rdst)$ and $\norm{g}_2=1$.
We use the notation from Section \ref{sec_tools}. The following lemma is one of the key observations in our proof.
\begin{lemma}
\label{lemma_tko}
Let $\Omega \subset \Rtdst$ be a compact set. Then the following formula holds
\begin{align*}
(1_\Omega * \env)(z)=\sum_{k \geq 1} \lambdako \abs{\eigf^\Omega_k(z)}^2,
\quad
z \in \Rtdst.
\end{align*}
\end{lemma}
\begin{proof}
Let $K(z,z')=\stft_g g(z-z') e^{2\pi i z_2 z'_1}$ be the reproducing kernel of $\VR$
and  $K_z  := \overline{K(z,\cdot)} \in \VR$. Hence for all $F \in \VR$,
\begin{align*}
F(z) = \int_{\Rtdst} K(z,z') F(z') \, dz'  = \ip{F}{K_z}, \qquad z \in \Rtdst.
\end{align*}
First we  compute
\begin{align*}
\ip{M_\Omega K_z}{K_z} &= \ip{1_\Omega K_z}{K_z}
\\
&= \int_{\Rtdst} 1_\Omega(z') K_z(z') \overline{K_z(z')} \, dz'
\\
&= \int_{\Rtdst} 1_\Omega(z') \env(z-z') \, dz'
= (1_\Omega * \env ) (z).
\end{align*}
Second, we use \eqref{eq_MO} to compute
\begin{align*}
\ip{M_\Omega K_z}{K_z}&=\ip{\sum_k \lambdako \ip{K_z}{\eigf^\Omega_k} \eigf^\Omega_k}{K_z}
\\
&=
\sum_k \lambdako \abs{\ip{\eigf^\Omega_k}{K_z}}^2
\\
&=
\sum_k \lambdako \abs{\eigf^\Omega_k(z)}^2.
\end{align*}
The lemma follows by equating the two formulae for $\ip{M_\Omega K_z}{K_z}$.
\end{proof}
We now use Lemma \ref{lemma_tko} to bound the error between
$\inten_\Omega$ and $1_\Omega$.
\begin{lemma}
\label{lemma_bound_eig}
Let $\Omega \subset \Rtdst$ be compact and set
\begin{align}
\label{eq_error}
\error(\Omega) := 1-\frac{\sum_{k=1}^{A_\Omega} \lambdako}{\mes{\Omega}}.
\end{align}
Then
\begin{align*}
\frac{1}{\mes{\Omega}}
\bignorm{\inten_{\Omega}- (1_{\Omega}*\env)}_1 \leq
\left( \frac{1}{\mes{\Omega}} + 2\error(\Omega) \right).
\end{align*}
\end{lemma}
\begin{proof}
Let us set $L_k := 1$ for $1 \leq k \leq A_{\Omega}$ 
and $L_k := 0$ for $k>  A_{\Omega}$. Using Lemma \ref{lemma_tko} we have that
\begin{align} \label{help}
\inten_{\Omega}(z)- (1_{\Omega}*\env)(z)
= \sum_{k\geq 1} (L_k - \lambda^{\Omega}_k) \abs{\eigf^{\Omega}_k(z)}^2.
\end{align}
Since $\bignorm{\abs{\eigf^{\Omega}_k}^2}_1=
\bignorm{V_g \eigenf_k^\Omega}_{2}^2 =\norm{\eigenf^{\Omega}_k}_2^2=1$,
we use \eqref{eq_trace1} and estimate
\begin{align*}
&\bignorm{\inten_{\Omega}- (1_{\Omega}*\env)}_1
\leq \left( \sum_{k\geq 1} \abs{L_k - \lambda^{\Omega}_k}\right)
= \left(\sum_{k=1}^{A_{\Omega}} (1 - \lambda^{\Omega}_k) + \sum_{k>A_{\Omega}} \lambda^{\Omega}_k\right)
\\
&\qquad = \left(A_{\Omega} - 2 \sum_{k=1}^{A_{\Omega}} \lambda^{\Omega}_k + \sum_{k\geq 1}
\lambda^{\Omega}_k\right)
= \left(A_{\Omega} - 2 \sum_{k=1}^{A_{\Omega}} \lambda^{\Omega}_k + \mes{\Omega}\right)
\\
&\qquad =\left(
(A_{\Omega} - \mes{\Omega}) + 2 ( \mes{\Omega} - \sum_{k=1}^{A_{\Omega}} \lambda^{\Omega}_k)
\right)
\leq
\left(1 + 2 ( \mes{\Omega} - \sum_{k=1}^{A_{\Omega}} \lambda^{\Omega}_k)\right)
\\
&\qquad = 1+2\error(\Omega)\mes{\Omega}.
\end{align*}
\end{proof}
The tail of the eigenvalue distribution $E(\Omega )$ can be estimated
as follows.
\begin{lemma}
\label{lemma_error}
Let $\Omega \subset \Rtdst$ be a compact set and consider 
the number $\error(\Omega)$ defined in \eqref{eq_error}. Then the following holds.
\begin{itemize}
\item[(a)] $\error(R \cdot \Omega) \longrightarrow 0$, as $R \longrightarrow +\infty$.
\item[(b)] If $\Omega$ has finite perimeter and $g \in \modstarsp$, then
\begin{align}
\label{eq_error_estimate}
0 \leq \error(\Omega) \leq 2 \norm{g}_\modstar \sqrt{\frac{\perim{\Omega}}{\mes{\Omega}}}.
\end{align}
\end{itemize}
\end{lemma}
\begin{proof}
First note that from \eqref{eq_trace1} we know that
$\sum_k \lambdako = \mes{\Omega}$ and consequently $0 \leq \error(\Omega) \leq 1$.

Let $\delta \in (0,1)$ and define $l_\delta(\Omega):=\min\{A_\Omega, \#\sett{k:\lambdako > 1-\delta}\}$.
Then
\begin{align*}
\lambdako \geq 1-\delta
\mbox{ for }1 \leq k \leq l_\delta(\Omega). 
\end{align*}
Since $A_\Omega \geq l_\delta(\Omega)$ we estimate
\begin{align*}
\sum_{k=1}^{A_\Omega} \lambdako \geq \sum_{k=1}^{l_\delta(\Omega)} \lambdako \geq (1-\delta) l_\delta(\Omega).
\end{align*}
Therefore,
\begin{align*}
0 \leq \error(\Omega) &\leq 1-(1-\delta) \frac{l_\delta(\Omega)}{\mes{\Omega}}.
\end{align*}
Since $A_\Omega \geq \mes{\Omega}$ we get
\begin{align}
\label{eq_error_2}
0 \leq \error(\Omega) \leq 1-(1-\delta)
\min\sett{1, \frac{\#\sett{k:\lambdako > 1-\delta}}{\mes{\Omega}}}.
\end{align}
To prove (a), we apply this estimate and Proposition \ref{prop_eigen_soft} to $R \cdot \Omega$ to deduce that
\begin{align*}
0 \leq \limsup_{R \rightarrow +\infty} \error(R \cdot \Omega)
\leq 1-(1-\delta) = \delta,
\end{align*}
and then let $\delta \longrightarrow 0^+$.

To prove (b), we apply Proposition \ref{prop_eigen} and obtain
\begin{align*}
\frac{\#\sett{k:\lambdako > 1-\delta}}{\mes{\Omega}}
\geq 1 - C_\delta \norm{g}_{\modstar}^2 \frac{\perim{\Omega}}{\mes{\Omega}},
\end{align*}
where $C_\delta =\max\{\delta^{-1},(1-\delta)^{-1}\}$. Combining this estimate with \eqref{eq_error_2}
gives
\begin{align*}
\error(\Omega) &\leq 1-(1-\delta)
\left(1 - C_\delta \norm{g}_{\modstar}^2 \frac{\perim{\Omega}}{\mes{\Omega}}\right)
\\
&=\delta+ (1-\delta)C_\delta \norm{g}_{\modstar}^2 \frac{\perim{\Omega}}{\mes{\Omega}}.
\end{align*}
Since $\delta \in (0,1)$, we have that $(1-\delta)C_\delta \leq 1/\delta$ and therefore
\begin{align}
\label{eq_error_delta}
\error(\Omega) &\leq \delta+ \frac{1}{\delta} \norm{g}_{\modstar}^2 \frac{\perim{\Omega}}{\mes{\Omega}}.
\end{align}
Finally, let $\delta := \norm{g}_{\modstar}\sqrt{\frac{\perim{\Omega}}{\mes{\Omega}}}$. Note
that we can assume that $\delta < 1$ since otherwise the bound in 
\eqref{eq_error_estimate} is trivial because $\error(\Omega) \leq 1$.
Therefore, we can apply \eqref{eq_error_delta} to get the desired conclusion.
\end{proof}
Finally we derive a  weak-$L^2$ estimate for the error
$\inten_\Omega - 1_\Omega * ~\Theta$. 
\begin{prop}
\label{prop_weak_l2}
Let $g \in \modstarsp(\Rdst)$ and let $\Omega \subset \Rtdst$ be a compact set with finite perimeter
and assume that $\norm{g}^2_\modstarsp \perim{\Omega} \geq 1$. Then
\begin{align}
\label{eq_prop_weak_l2}
\Bigabs{\sett{z\in \mathbb{R}^{2d}: \bigabs{\inten_\Omega (z)
      -(1_\Omega*\env)(z) }>\delta}}
\lesssim \frac{1}{\delta^2}\norm{g}_{\modstar}^2 \perim{\Omega},
\qquad \delta>0.
\end{align}
\end{prop}
\begin{proof}
Let $\delta \in (0,1/2]$ and set
\begin{align*}
&a_\delta := \#\sett{k:\lambdako > 1-\delta},
\\
&b_\delta := \#\sett{k:\lambdako > \delta}.
\end{align*}
Since $\delta \leq 1/2$, $a_\delta \leq b_\delta$. From Proposition \ref{prop_eigen}
(using again that $\delta \leq 1/2$) we obtain
\begin{align*}
&a_\delta \geq \mes{\Omega} - \frac{1}{\delta} \norm{g}_{\modstar}^2 \perim{\Omega},
\\
&b_\delta \leq \mes{\Omega} + \frac{1}{\delta} \norm{g}_{\modstar}^2 \perim{\Omega}.
\end{align*}
Since $0\leq A_\Omega-\mes{\Omega} \leq 1$,
\begin{align*}
&a_\delta \geq A_\Omega - 1 - \frac{1}{\delta} \norm{g}_{\modstar}^2 \perim{\Omega},
\\
&b_\delta \leq A_\Omega + \frac{1}{\delta} \norm{g}_{\modstar}^2 \perim{\Omega}.
\end{align*}
Next  set $a'_\delta := \min\{a_\delta, A_\Omega\}$ and
$b'_\delta := \max\{b_\delta, A_\Omega\}$. Then the size of the plunge
region $\{ k \in \mathbb{N}: \delta < \lambda _k^\Omega \leq 1-\delta
\}$ is bounded by 
\begin{align*}
0 \leq b'_\delta - a'_\delta \leq \frac{2}{\delta} \norm{g}_{\modstar}^2 \perim{\Omega} + 1.
\end{align*}
Using the fact that $\norm{g}^2_\modstarsp \perim{\Omega} \geq 1$ we obtain
\begin{align}
\label{eq_bound_dis}
0 \leq b'_\delta - a'_\delta \lesssim \frac{1}{\delta} \norm{g}_{\modstar}^2 \perim{\Omega}.
\end{align}
Let us define $\sett{\mu_k: k \geq 1}$ by $\mu_k := 1-\lambdako$ for $k \leq A_\Omega$
and $\mu_k=\lambdako$ for $k>A_\Omega$. Then, by Lemma \ref{lemma_tko},
\begin{align*}
\abs{\inten_\Omega(z)-1_\Omega*\env(z)} \leq  \sum_{k\geq 1}
\mu_k  \abs{\eigf_k(z)}^2,
\qquad z \in \Rtdst.
\end{align*}
Since  $0 \leq \mu_k \leq 1$ and $0 \leq \mu_k \leq \delta$
if either $k \leq a'_\delta$ or $k > b'_\delta$,  \eqref{eq_sum_one}
implies that 
\begin{align*}
\bigabs{\inten_\Omega(z)-1_\Omega*\env(z)} &\leq
\sum_{k=1}^{a'_\delta} \mu_k \abs{\eigf_k(z)}^2
+
\sum_{k=a'_\delta+1}^{b'_\delta} \mu_k \abs{\eigf_k(z)}^2
+
\sum_{k>b'_\delta} \mu_k \abs{\eigf_k(z)}^2
\\
&\leq 2\delta + \sum_{k=a'_\delta+1}^{b'_\delta} \abs{\eigf_k(z)}^2.
\end{align*}
We set  $f_\delta := \sum_{k=a'_\delta+1}^{b'_\delta} \abs{\eigf_k(z)}^2$
and use \eqref{eq_bound_dis} to bound
\begin{align*}
&\bigabs{\sett{\abs{\inten_\Omega-1_\Omega*\env} \geq 3 \delta}}
\leq \bigabs{\sett{\abs{f_\delta} \geq \delta}}
\leq \frac{1}{\delta} \norm{f_\delta}_1
\\
&\qquad \lesssim \frac{1}{\delta^2}\norm{g}_{\modstar}^2 \perim{\Omega}.
\end{align*}
Making the change of variables $\delta \mapsto \delta/3$, we
obtain \eqref{eq_prop_weak_l2}
for $0 < \delta \leq 3/2$.
 Finally note that  \eqref{eq_prop_weak_l2} is trivial for $\delta >1$
 because  
\begin{align*}
\abs{\inten_\Omega(z)-1_\Omega*\env(z)} \leq \sum_{k\geq 1} \mu_k \abs{\eigf_k(z)}^2
\leq \sum_{k\geq 1} \abs{\eigf_k(z)}^2 \leq 1.
\end{align*}

\end{proof}
\section{Proof of the main results}
\label{sec_res}
We now combine the bounds from Section \ref{sec_bounds} and derive our
main  estimates on 
the accumulated spectrogram $\inten_\Omega$.
First we recall and prove Theorem \ref{th_non_asym}.
\begin{reptheorem}{th_non_asym}
Assume that $g \in \modstarsp(\Rdst)$ with $\norm{g}_2=1$
and that $\Omega \subset \Rtdst$ is a compact set with finite perimeter.
Then
\begin{align*}
\frac{1}{\mes{\Omega}} 
\bignorm{\inten_{\Omega}- 1_{\Omega}*\abs{V_g g}^2}_1 \leq 
\left(
\frac{1}{\mes{\Omega}}+4\norm{g}_\modstar \sqrt{\frac{\perim{\Omega}}{\mes{\Omega}}}
\right).
\end{align*}
\end{reptheorem}
\begin{proof}
The theorem follows immediately by combining Lemmas \ref{lemma_bound_eig} and \ref{lemma_error}.
\end{proof}

Theorem \ref{th_non_asym} provides an estimate for the accumulated spectrogram by the smoothed function
$1_\Omega * \env$. The following corollary estimates directly the error between $\inten_\Omega$ and $1_\Omega$.
\begin{coro}
\label{coro_error}
Let $g \in \modstarsp(\Rdst)$ and let $\Omega \subset \Rtdst$ be a compact set with finite perimeter.
Then 
\begin{align*}
\frac{1}{\mes{\Omega}} \norm{\inten_{\Omega}-1_{\Omega}} _1
\leq \frac{1}{\mes{\Omega}}+\norm{g}_\modstar^2 \frac{\perim{\Omega}}{\mes{\Omega}}+
4\norm{g}_\modstar \sqrt{\frac{\perim{\Omega}}{\mes{\Omega}}}.
\end{align*}
In particular, for $\mes{\Omega} \geq 1$
\begin{align}
\label{eq_coro_error}
\frac{1}{\mes{\Omega}} \norm{\inten_{\Omega}-1_{\Omega}} _1
\lesssim \sqrt{\frac{\perim{\Omega}}{\mes{\Omega}}},
\end{align}
where the implicit constant depends on the window $g$.
\end{coro}
\begin{proof}
For the first part, we simply estimate
\begin{align*}
\bignorm{\inten_{\Omega}- 1_{\Omega}}_1
\leq \bignorm{\inten_{\Omega}- (1_{\Omega}*\env)}_1
+ \bignorm{(1_{\Omega}*\env)-1_{\Omega}}_1,
\end{align*}
and apply Theorem \ref{th_non_asym} and Lemma \ref{lemma_var}. For the second part, note that
$\norm{\inten_\Omega}_1 = A_\Omega = \mes{\Omega} + O(1)$ and consequently the left-hand side in
\eqref{eq_coro_error} is $\lesssim 1$. This allows us to assume that
$\frac{\perim{\Omega}}{\mes{\Omega}} \leq 1$. Consequently,
$\sqrt{\frac{\perim{\Omega}}{\mes{\Omega}}}$ dominates both
$\frac{\perim{\Omega}}{\mes{\Omega}}$ and $\frac{1}{\mes{\Omega}}$ and the conclusion follows.
\end{proof}
\begin{rem}
{\rm Since $\norm{\inten_{\Omega}-1_{\Omega}}_\infty \leq 2$, (complex)
interpolation and Corollary \ref{coro_error} imply  the following
$L^p$-estimate (when  $|\Omega | \geq
1$):}
\begin{align}
\label{eq_coro_lp}
\frac{1}{\mes{\Omega}} \bignorm{\inten_{\Omega}-1_{\Omega}}_p
\lesssim \frac{\perim{\Omega}^{1/(2p)}}{\mes{\Omega}^{1-1/(2p)}},
\qquad 1 \leq p \leq +\infty.
\end{align}
\end{rem}

We now prove Theorem \ref{th_one_point} that shows the
asymptotic convergence for the family of dilations of a single set.
\begin{reptheorem}{th_one_point}
Let $g \in L^2(\Rdst)$, $\norm{g}_2=1$, and let $\Omega \subset \Rtdst$ be compact. Then
\begin{align*}
\inten_{R\cdot\Omega} (R \cdot) \longrightarrow 1_\Omega \,\, 
\mbox{ in $L^1(\Rtdst)$, as }R\longrightarrow +\infty.
\end{align*} 
\end{reptheorem}
\begin{proof}
For $R>0$, let $\env_R(z) := R^{2d}\env(R z)$. Then
\begin{align*}
(1_{R\cdot\Omega}*\env)(R z)= (1_\Omega * \env_R)(z), \qquad z \in \Rtdst.
\end{align*}
Let us estimate
\begin{align*}
&\int_{\Rtdst} \abs{\inten_{R \cdot \Omega}(R z)- 1_{\Omega}(z)}dz
\\
&\qquad \leq
\int_{\Rtdst} \abs{\inten_{R \cdot \Omega}(R z)- (1_{\Omega}*\env_R)(z)}dz
+
\int_{\Rtdst} \abs{(1_{\Omega}*\env_R)(z)-1_\Omega(z)}dz
\\
&\qquad =
\int_{\Rtdst} \abs{\inten_{R \cdot \Omega}(R z)- (1_{R\cdot\Omega}*\env)(R z)}dz
+
\int_{\Rtdst} \abs{(1_{\Omega}*\env_R)(z)-1_\Omega(z)}dz
\\
&\qquad =
\frac{\mes{\Omega}}{\mes{R\cdot\Omega}}
\int_{\Rtdst} \abs{\inten_{R \cdot \Omega}(z)- (1_{R\cdot\Omega}*\env)(z)}dz
+
\int_{\Rtdst} \abs{(1_{\Omega}*\env_R)(z)-1_\Omega(z)}dz.
\end{align*}
Hence, applying Lemma \ref{lemma_bound_eig} to the first term we obtain
\begin{align*}
&\int_{\Rtdst} \abs{\inten_{R \cdot \Omega}(R z)- 1_{\Omega}(z)}dz
\\
&\qquad \leq \mes{\Omega} \left(\frac{1}{\mes{R\cdot\Omega}} +2\error(R\cdot\Omega) \right)+
\int_{\Rtdst} \abs{1_\Omega(z)- (1_{\Omega}*\env_R)(z)}dz.
\end{align*}
By Lemma \ref{lemma_error}, $\error(R\cdot\Omega) \longrightarrow 0$, 
as $R \longrightarrow +\infty$. In addition, since $\int \env =1$, 
$\env_R$ is an approximate identity in $L^1$, and
consequently $1_\Omega * \env_R \longrightarrow 1_\Omega$ in $L^1$ for
$R\to +\infty$. This completes the proof.
\end{proof}
Finally, we derive a weak-$L^2$ estimate, that, as opposed to Corollary \ref{coro_error},
provides an error bound that only depends on $\perim{\Omega}$.

\begin{reptheorem}{th_weak_l2}
Let $g \in \modstarsp(\Rdst)$ and let $\Omega \subset \Rtdst$ be a compact set with finite perimeter
and assume that $\norm{g}^2_\modstarsp \perim{\Omega} \geq 1$. Then
\begin{align*}
\Bigabs{\sett{ z\in \mathbb{R}^{2d}: \bigabs{\inten_\Omega(z)-1_\Omega(z)}>\delta}}
\lesssim \frac{1}{\delta^2}\norm{g}_{\modstar}^2 \perim{\Omega},
\qquad \delta>0.
\end{align*}
\end{reptheorem}
\begin{proof}
Let $\delta>0$. Since $\norm{\inten_\Omega-1_\Omega}_\infty \leq 2$, we assume without loss of generality
that $\delta \leq 2$.
Using Proposition \ref{prop_weak_l2} and Lemma \ref{lemma_var} 
we estimate
\begin{align*}
\Bigabs{\sett{\bigabs{\inten_\Omega-1_\Omega}>\delta}}
&\leq
\Bigabs{\sett{\bigabs{\inten_\Omega-1_\Omega*\env}>\delta/2}}
+
\Bigabs{\sett{\bigabs{1_\Omega*\env-1_\Omega}>\delta/2}}
\\
&\leq \Bigabs{\sett{\bigabs{\inten_\Omega-1_\Omega*\env}>\delta/2}}
+
\frac{2}{\delta}\bignorm{1_\Omega*\env-1_\Omega}_1
\\
&\lesssim \frac{1}{\delta^2}\norm{g}_{\modstar}^2 \perim{\Omega}
+\frac{1}{\delta}\norm{g}_{\modstar}^2 \perim{\Omega}.
\end{align*}
Since $\delta \leq 2$, $1/\delta \leq 2 /\delta^2$ and the conclusion follows.
\end{proof}

\section{Approximate retrieval of time-frequency filters}
\label{sec_app}
In signal processing, the time-frequency localization operators
$\locom$ are also called \emph{time-frequency filters}. Whereas the classical time-invariant filters multiply
the Fourier transform of a signal by a given symbol, time-frequency filters 
localize signals both in time and frequency and are therefore time-varying. The field of system identification
studies the possibility of retrieving a linear  operator from
its response to a set of  test signals. For time-varying systems the
identification problem is particularly difficult \cite{be69, ka62, pf08-1, pfwa06}.

In the case of a time-frequency filter, it is also important to understand to what extent the operator
$\locom$ can be understood from the measurement of a few of its eigenmodes (eigenfunctions)
$h^\Omega_1, \ldots, h^\Omega_n$. In
\cite{abdo12} the following special case was established (see \cite{abdo12} for a discussion on possible applications).
\begin{theo}
\label{th_abdo}
Let $g(t) := 2^{1/4} e^{-\pi t^2}$, $t \in \Rst$, be the one-dimensional Gaussian and let $\Omega \subseteq \Rst^2$
be compact and simply connected. If one of the eigenfunctions of
$\locom$ is a Hermite function, then $\Omega$ is a disk centered at
$0$. 
\end{theo}
Hence, for a Gaussian window, $\Omega $ is completely determined by
the information that (a) $\Omega \subseteq \Rtdst$ is a simply
connected with given measure $|\Omega |$   and that (b)  one of the  eigenfunctions of
$\locom $  is a Hermite function. 

However, Theorem \ref{th_abdo} has some drawbacks. First, it is non-robust: from the information that the eigenmodes of
$\locom$ look approximately like Hermite functions we cannot conclude that $\Omega$ is approximately a disk. Second,
it only applies to the restricted situation of a one-dimensional Gaussian window. Both restrictions stem from the
one-variable complex analysis techniques used in \cite{abdo12}.

While the exact and robust recovery of the fine details of a time-frequency filter may not be possible
only from measurements of a few of its eigenmodes, we may  recover at least the coarse shape of the
set by means of the accumulated spectrogram.
Let us consider a time-frequency filter $\locom$ and suppose that we know
the measure of $\Omega$ and the \emph{spectrogram} of the first eigenmodes
$\abs{V_g h^\Omega_1}^2, \ldots, \abs{V_g h^\Omega_{A_\Omega}}^2$, $A_\Omega:=\ceil{\mes{\Omega}}$. Then we approximate
the (unknown) domain $\Omega $ by the level sets of the accumulated spectrogram $\rho_\Omega$.
\begin{theo}
\label{th_recovery}
Let $g \in \modstarsp(\Rdst)$ and let $\Omega \subset \Rtdst$ be a compact set with finite perimeter
and $\norm{g}^2_\modstarsp \perim{\Omega} \geq 1$. Let
\begin{align}
\label{eq_approx_omega}
\widetilde{\Omega} := \set{z \in \Rtdst}{\inten_{\Omega}(z) > 1/2}.
\end{align}
Then
\begin{align*}
\mes{\Omega \triangle \widetilde{\Omega}}
\lesssim \norm{g}^2_\modstar \perim{\Omega},
\end{align*}
where $\triangle$ denotes the symmetric difference of two sets. 
\end{theo}
\begin{proof}
Let $E_\delta := \set{z \in \Rtdst}{\abs{1_\Omega(z)-\inten_{\Omega}(z)} \geq 1/2}$.
Let us note that
\begin{align}
\label{eq_inter}
\Omega \cap (\Rtdst \setminus E_\delta) = \widetilde{\Omega} \cap (\Rtdst \setminus E_\delta).
\end{align}
Indeed, if $z \in \Omega \cap (\Rtdst \setminus E_\delta)$, then
$\inten_\Omega(z) \geq 1_\Omega(z) - \abs{1_\Omega(z)-\inten_{\Omega}(z)} > 1-1/2=1/2$. Hence
$z \in \widetilde{\Omega}$. Second, if $z \in \widetilde{\Omega} \cap (\Rtdst \setminus E_\delta)$,
then $1_\Omega(z) \geq \inten_\Omega(z) - \abs{1_\Omega(z)-\inten_{\Omega}(z)} > 1/2 - 1/2 = 0$.\
Hence $z \in \Omega$.

The equality in \eqref{eq_inter} simply means that
$\Omega \triangle \widetilde{\Omega} \subseteq E_\delta$.
To finish the proof,
we apply Theorem \ref{th_weak_l2} to bound the measure of $E_\delta$:
\begin{align*}
\mes{E_\delta} \lesssim \norm{g}^2_\modstar\perim{\Omega}.
\end{align*}
\end{proof}
Finally we observe that the approximation $\widetilde{\Omega}$ does not require 
the phases of the short-time Fourier transforms
$V_g h^\Omega_1, \ldots, V_g h^\Omega_{A_\Omega}$ but only their absolute
values. This fact is very valuable in applications and is referred to as
phase retrieval (see for example \cite{babocaed09, mi14}). 

\section{Acknowledgment}
The authors thank  Hans Feichtinger for useful discussion and  for
suggesting the term accumulated spectrogram. 

\bibliographystyle{abbrv}

\begin{thebibliography}{10}

\bibitem{abdo12}
L.~D. {A}breu and M.~{D}{\"o}rfler.
\newblock {A}n inverse problem for localization operators.
\newblock {\em Inverse Problems}, 28(11):115001, 16, 2012. 
 
 \bibitem{babocaed09}
R.~M. {B}alan, B.~G. {B}odmann, P.~G. {C}asazza, and D.~{E}didin.
\newblock {P}ainless reconstruction from magnitudes of frame coefficients.
\newblock {\em J. Fourier Anal. Appl.}, 15(4):488--501, 2009. 

\bibitem{be69}
P.~A. {B}ello.
\newblock {M}easurement of random time-variant linear channels.
\newblock {\em IEEE Trans. Inform. Theory}, 15(4):469--475, 1969.

\bibitem{DetPointRand} J. {B}en {H}ough, M. {K}rishnapur, Y. {P}eres, B. {V}ir\'{a}g,
\newblock Zeros of Gaussian Analytic Functions and Determinantal Point
Processes.
\newblock University Lecture Series Vol. 51, x+154, American Mathematical
Society, Providence, RI (2009).

 \bibitem{bocogr04}
P.~{B}oggiatto, E.~{C}ordero, and K.~{G}r{\"o}chenig.
\newblock {G}eneralized anti-{W}ick operators with symbols in distributional
{S}obolev spaces.
\newblock {\em Integr. Equ. Oper. Theory}, 48(4):427--442, 2004. 

\bibitem{cogr03}
E.~{C}ordero and K.~{G}r{\"o}chenig.
\newblock {T}ime-frequency analysis of localization operators.
\newblock {\em J. Funct. Anal.}, 205(1):107--131, 2003. 

\bibitem{da88}
I.~{D}aubechies.
\newblock {T}ime-frequency localization operators: a geometric phase space
  approach.
\newblock {\em IEEE Trans. Inform. Theory}, 34(4):605--612, {J}uly 1988.

\bibitem{da90}
I.~{D}aubechies.
\newblock {T}he wavelet transform, time-frequency localization and signal
  analysis.
\newblock {\em IEEE Trans. Inform. Theory}, 36(5):961--1005, 1990.

\bibitem{de07-7}
P.~{D}eift.
\newblock {U}niversality for mathematical and physical systems.
\newblock In {\em {I}nternational {C}ongress of {M}athematicians. {V}ol. {I}},
pages 125--152. {E}ur. {M}ath. {S}oc., {Z}{\"u}rich, 2007. 

\bibitem{defeno02}
F.~{D}e{M}ari, H.~G. {F}eichtinger, and K.~{N}owak.
\newblock {U}niform eigenvalue estimates for time-frequency localization
  operators.
\newblock {\em J. London Math. Soc.}, 65(3):720--732, 2002.

 \bibitem{duwozh01}
J.~{D}u, M.~{W}ong, and Z.~{Z}hang.
\newblock {T}race class norm inequalities for localization operators.
\newblock {\em Integr. Equ. Oper. Theory}, 41(4):497--503, 2001. 

\bibitem{en09}
M.~{E}nglis.
\newblock {T}oeplitz operators and localization operators.
\newblock {\em Trans. Amer. Math. Soc.}, 361(2):1039--1052, 2009. 

\bibitem{evga92}
L.~C. {E}vans and R.~F. {G}ariepy.
\newblock {\em {M}easure {T}heory and {F}ine {P}roperties of {F}unctions.}
\newblock {S}tudies in {A}dvanced {M}athematics. {C}{R}{C} {P}ress, {B}oca
  {R}aton, 1992.

\bibitem{fe06}
H.~G. {F}eichtinger.
\newblock {M}odulation {S}paces: {L}ooking {B}ack and {A}head.
\newblock {\em Sampl. Theory Signal Image Process.}, 5(2):109--140, 2006.

\bibitem{feno01}
H.~G. {F}eichtinger and K.~{N}owak.
\newblock {A} {S}zeg{\"o}-type theorem for {G}abor-{T}oeplitz localization
  operators.
\newblock {\em Michigan Math. J.}, 49(1):13--21, 2001.

\bibitem{GG02}
Y.~V. Galperin and K.~Gr{\"o}chenig.
\newblock Uncertainty principles as embeddings of modulation spaces.
\newblock {\em J. Math. Anal. Appl.}, 274(1):181--202, 2002.

\bibitem{gi65}
J.~{G}inibre.
\newblock {S}tatistical ensembles of complex, quaternion, and real matrices.
\newblock {\em J. Mathematical Phys.}, 6:440--449, 1965. 

\bibitem{gr01}
K.~{G}r{\"o}chenig.
\newblock {\em {F}oundations of {T}ime-{F}requency {A}nalysis}.
\newblock {A}ppl. {N}umer. {H}armon. {A}nal. {B}irkh{\"a}user {B}oston,
  {B}oston, {M}{A}, 2001.

\bibitem{herato94}
C.~{H}eil, J.~{R}amanathan, and P.~{T}opiwala.
\newblock {\em {A}symptotic {S}ingular {V}alue {D}ecay of {T}ime-frequency
  {L}ocalization {O}perators}, 
\newblock
in "Wavelet Applications in Signal and Image Processing II, Proc. SPIE", Vol.2303 (1994)
p.15--24.

\bibitem{herato97}
C.~{H}eil, J.~{R}amanathan, and P.~{T}opiwala.
\newblock {S}ingular values of compact pseudodifferential operators.
\newblock {\em J. Funct. Anal.}, 150(2):426--452, 1997. 

\bibitem{ka62}
T.~{K}ailath.
\newblock {M}easurements on time-variant communication channels.
\newblock {\em IEEE Trans. Inform. Theory}, 8(5):229-- 236, 1962.

\bibitem{la67-1}
H.~J. {L}andau.
\newblock {S}ampling, data transmission, and the {N}yquist rate.
\newblock {\em Proc. IEEE}, 55(10):1701--1706, {O}ctober 1967.

\bibitem{la67}
H.~J. {L}andau.
\newblock {N}ecessary density conditions for sampling an interpolation of
  certain entire functions.
\newblock {\em Acta Math.}, 117:37--52, 1967.

\bibitem{la75-1}
H.~J. {L}andau.
\newblock {O}n {S}zeg{\"o}'s eigenvalue distribution theorem and
  non-{H}ermitian kernels.
\newblock {\em J. Anal. Math.}, 28:335--357, 1975.

\bibitem{lapo61}
H.~J. {L}andau and H.~O. {P}ollak.
\newblock {P}rolate spheroidal wave functions, {F}ourier analysis and
uncertainty {I}{I}.
\newblock {\em Bell System Tech. J.}, 40:65--84, 1961.

\bibitem{lapo62}
H.~J. {L}andau and H.~O. {P}ollak.
\newblock {P}rolate spheroidal wave functions, {F}ourier analysis and
uncertainty {I}{I}{I}: {T}he dimension of the space of essentially time- and
band-limited signals.
\newblock {\em {B}ell {S}ystem {T}ech. {J}.}, 41:1295--1336, 1962. 

\bibitem{lawi80}
H.~J. {L}andau and H.~{W}idom.
\newblock {E}igenvalue distribution of time and frequency limiting.
\newblock {\em J. Math. Anal. Appl.}, 77:469--481, 1980.

\bibitem{li90-1}
E.~H. {L}ieb.
\newblock {I}ntegral bounds for radar ambiguity functions and {W}igner
  distributions.
\newblock {\em J. Math. Phys.}, 31(3):594--599, 1990.

\bibitem{mi14}
D.~G. {M}ixon.
\newblock {P}hase Transitions in Phase Retrieval.
\newblock arXiv:1403.1458.
 
\bibitem{pf08-1}
G.~E. {P}fander.
\newblock {M}easurement of time-varying multiple-input multiple-output
  channels.
\newblock {\em Appl. Comput. Harmon. Anal.}, 24(3):393--401, 2008.

\bibitem{pfwa06}
G.~E. {P}fander and D.~F. {W}alnut.
\newblock {M}easurement of time-variant channels.
\newblock {\em IEEE Trans. Inform. Theory}, 52(11):4808--4820, {N}ovember 2006.

\bibitem{rato93}
J.~{R}amanathan and P.~{T}opiwala.
\newblock {T}ime-frequency localization via the {W}eyl correspondence.
\newblock {\em SIAM J. Math. Anal.}, 24(5):1378--1393, 1993.

\bibitem{rato94}
J.~{R}amanathan and P.~{T}opiwala.
\newblock {T}ime-frequency localization and the spectrogram.
\newblock {\em Appl. Comput. Harmon. Anal.}, 1(2):209--215, 1994.

\bibitem{rito12}
B.~{R}icaud and B.~{T}orr{\'e}sani.
\newblock {A} survey of uncertainty principles and some signal processing
  applications.
\newblock {\em Adv. Comput. Math.}, to appear. DOI: 10.1007/s10444-013-9323-2.

\bibitem{si79}
B.~{S}imon.
\newblock {\em {T}race {I}deals and their {A}pplications.}
\newblock {C}ambridge {U}niversity {P}ress, {C}ambridge, 1979.

\bibitem{posl61}
D.~{S}lepian and H.~O. {P}ollak.
\newblock {P}rolate {S}pheroidal {W}ave {F}unctions, {F}ourier {A}nalysis and
{U}ncertainty {I}.
\newblock {\em {I}. {B}ell {S}yst. {T}ech.{J}.}, 40(1):43--63, 1961. 

\end{thebibliography}

\end{document}